\documentclass[11pt, one side]{amsart}
\usepackage[a4paper,margin=1in]{geometry}
\usepackage{fancyhdr}
\usepackage{enumerate}
\usepackage[linktocpage=true]{hyperref}
\usepackage[centertags]{amsmath}
\usepackage{amssymb,latexsym,xy,eucal,mathrsfs}
\newtheorem{theorem}{Theorem}[section]
\newtheorem{lemma}[theorem]{Lemma}

\newtheorem{maintheorem}[theorem]{Main Theorem}
\newtheorem{corollary}[theorem]{Corollary}
\newtheorem{definition}[theorem]{Definition}

\usepackage{enumerate}
\newtheorem{proposition}[theorem]{Proposition}
\numberwithin{equation}{section} 
\large
\begin{document}

\title{Connectivity of single-element coextensions of a binary matroid}
%\date{}
\author{Ganesh Mundhe$^1$ and Y. M. Borse$^2$}
%\address{\rm  Department of Mathematics, SPPU, Pune-411007, INDIA. }

\address{\rm 1. Army Institute of Technology, Pune-411015, INDIA.}
\email{{ ganumundhe@gmail.com }}
\address{\rm 2. Department of Mathematics, Savitribai Phule  Pune University, Pune-411007, INDIA.}

\email{{ ymborse11@gmail.com }}
 
 \maketitle{}
 \noindent
\begin{abstract}  Given an $n$-connected binary matroid, we obtain a necessary and sufficient condition for its single-element coextensions to be $n$-connected.
\end{abstract}


\section{Introduction}

For undefined terminologies, we refer to Oxley \cite{oxley2006matroid}. The point-splitting operation is a fundamental operation  in respect of connectivity of graphs.  It is used to characterize 3-connected graphs in  the classical Tutte's Wheel Theorem \cite{tutte1961theory} and also to characterize 4-connected graphs by Slater \cite{slater1974classification}. This operation is defined as follows.
\begin{definition} [\cite{slater1974classification}] \label{nps}
	Let $G$ be a graph with a vertex $v$ of degree at least $ 2n-2$ and let $ T = \{ vv_1, vv_2,  \dots, vv_{n-1}\}$ be a set of $n-1$ edges of $G$ incident to $v$. Let $G_T'$ be the graph obtained from $G$ by replacing $v$ by two adjacent vertices $u$ and $w$ such that $u$ is adjacent to $v_1, v_2, \dots, v_{n-1},$ and $w $ is adjacent to the vertices which are adjacent to $v$ except $v_1, v_2, \dots, v_{n-1}$.  We say $G_T'$ arises from $G$ by $n$-point splitting (see the following figure). 
\end{definition}

%The following example provides an illustration for this operation.
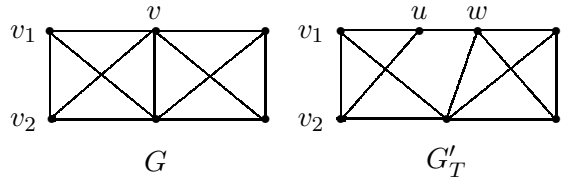
\begin{center} 
%TeXCAD (http://texcad.sf.net/) Picture. File: [Point splitting.pic]. Options on following lines.
%\grade{\on}
%\emlines{\off}
%\epic{\off}
%\beziermacro{\on}
%\reduce{\on}
%\snapping{\off}
%\pvinsert{% Your \input, \def, etc. here}
%\quality{8.000}
%\graddiff{0.005}
%\snapasp{1}
%\zoom{4.0000}
\unitlength .9mm % = 2.561pt
\linethickness{0.4pt}
\ifx\plotpoint\undefined\newsavebox{\plotpoint}\fi % GNUPLOT compatibility
\begin{picture}(80.225,23.75)(0,0)
\put(5.81,8.12){\circle*{1.33}}
\put(48.06,8.12){\circle*{1.33}}
\put(21.06,8.12){\circle*{1.33}}
\put(63.56,8.12){\circle*{1.33}}
\put(21.06,21.12){\circle*{1.33}}
\put(59.56,21.12){\circle*{1.33}}
\put(68.06,21.12){\circle*{1.33}}
\put(37.06,8.12){\circle*{1.33}}
\put(79.56,8.12){\circle*{1.33}}
\put(37.06,21.12){\circle*{1.33}}
\put(79.56,21.12){\circle*{1.33}}
\put(21.06,8.12){\line(-1,0){15}}
\put(21.06,21.12){\line(1,0){16}}
\put(37.06,21.12){\line(0,-1){13}}
\put(20.81,20.87){\line(0,-1){13}}
\put(79.56,21.12){\line(0,-1){13}}
\put(37.06,8.12){\line(-1,0){16}}
\put(79.56,8.12){\line(-1,0){16}}
\qbezier(6.06,8.12)(6.56,8.75)(6.06,7.87)
\qbezier(48.56,8.12)(49.06,8.75)(48.56,7.87)
\qbezier(6.06,7.87)(5.94,8.87)(6.31,7.87)
\qbezier(6.31,7.87)(5.44,9.12)(6.06,8.37)
\put(5.56,21.12){\circle*{1.33}}
\put(48.06,21.12){\circle*{1.33}}
\put(5.56,21.12){\line(1,0){16}}
%\emline(21.56,21.12)(20.81,21.37)
\multiput(21.56,21.12)(-.107143,.035714){7}{\line(-1,0){.107143}}
%\end
\put(20.81,21.37){\line(1,0){.25}}
%\emline(20.81,8.37)(5.31,21.12)
\multiput(20.81,8.37)(-.0454545455,.0373900293){341}{\line(-1,0){.0454545455}}
%\end
%\emline(63.31,8.37)(47.81,21.12)
\multiput(63.31,8.37)(-.0454545455,.0373900293){341}{\line(-1,0){.0454545455}}
%\end
\put(5.56,21.12){\line(0,-1){12.75}}
\put(48.06,21.12){\line(0,-1){12.75}}
%\emline(20.81,21.37)(5.81,8.37)
\multiput(20.81,21.37)(-.0432276657,-.0374639769){347}{\line(-1,0){.0432276657}}
%\end
%	\put(42,-5){\makebox(0,0)[cc]{Figure 1}}
%\emline(21,21)(36.75,8)
\multiput(21,21)(.045389049,-.0374639769){347}{\line(1,0){.045389049}}
%\end
%\emline(21,8)(37,21)
\multiput(21,8)(.0461095101,.0374639769){347}{\line(1,0){.0461095101}}
%\end
%\emline(63.5,8)(79.5,21)
\multiput(63.5,8)(.0461095101,.0374639769){347}{\line(1,0){.0461095101}}
%\end
\put(48,21.25){\line(1,0){11.5}}
%\emline(59.5,21.25)(48.25,8.25)
\multiput(59.5,21.25)(-.0373754153,-.0431893688){301}{\line(0,-1){.0431893688}}
%\end
\put(79.5,21){\line(-1,0){.25}}
%\emline(68,21.25)(79.75,7.75)
\multiput(68,21.25)(.0374203822,-.0429936306){314}{\line(0,-1){.0429936306}}
%\end
\put(79.75,7.75){\line(0,1){.25}}
\put(20.75,23.75){\makebox(0,0)[cc]{$v$}}
\put(59.5,23.75){\makebox(0,0)[cc]{$u$}}
\put(68,23.75){\makebox(0,0)[cc]{$w$}}
\put(59.25,21.25){\line(1,0){8.75}}
\put(63.75,18.75){\makebox(0,0)[cc]{}}
\put(1.75,20.75){\makebox(0,0)[cc]{$v_1$}}
\put(1.75,7.75){\makebox(0,0)[cc]{$v_2$}}
\put(43.75,20.75){\makebox(0,0)[cc]{$v_1$}}
\put(43.75,7.75){\makebox(0,0)[cc]{$v_2$}}
\put(21,2){\makebox(0,0)[cc]{$G$}}
\put(63.5,2){\makebox(0,0)[cc]{$G_T'$}}
%\put(63.25,23.5){\makebox(0,0)[cc]{$a$}}
%\emline(68,21.25)(63.5,8.25)
\multiput(68,21.25)(-.037190083,-.107438017){121}{\line(0,-1){.107438017}}
%\end
\put(67.75,21.25){\line(1,0){11.5}}
\put(63.5,8.25){\line(-1,0){16}}
\end{picture}

\end{center} 

Slater \cite{slater1974classification} obtained the following result to characterize   $4$-connected graphs.
\begin{theorem} [\cite{slater1974classification}]  \label{c5slater}
	Let $G$ be an $n$-connected graph and let $T$ be a set of $n-1$ edges incident to a vertex of degree at least $2n-2$.  Then the graph $G'_T$ is $n$-connected. 
\end{theorem}

In this paper, we extend the above theorem to binary matroids.

Azadi \cite{azadi2001generalized} extended the $n$-point splitting operation on graphs to binary matroids as follows. 
\begin{definition} [\cite{azadi2001generalized}] \label{defel}
 Let $M$ be a binary matroid with standard matrix representation $A$ over the field $GF(2)$ and let $T$ be a subset of the ground set  $E(M)$ of $M$.  Let $A'_T$ be the matrix obtained from $A$ by adjoining one extra row to matrix $A$ whose entries are 1 in the columns labeled by the elements of $T$ and 0 otherwise and also having one extra column labeled by $a$ with 1 in the last row and 0 elsewhere. Denote the vector matroid of $A'_T$  by $M'_T.$   We say that $M_T'$ is obtained from $M$ by element splitting with respect to the set $T$.
\end{definition}

For example, the following matrices $A$ and $A_T'$ represent the Fano matroid $F_7$ and its element splitting matroid with respect to  the set $T=\{1,2,3\} \subset E(F_7).$ 

\begin{center} 
 	$A = \bordermatrix{  ~&1&2&3&4&5&6&7\cr
 	                     ~&1&0&0&1&1&0&1\cr
 	                     ~&0&1&0&1&1&1&0\cr
 	                     ~&0&0&1&1&0&1&1},$
 	                                         $A'_T = \bordermatrix{     ~&1&2&3&4&5&6&7&a\cr
						 	                     	~&1&0&0&1&1&0&1&0\cr
						 	                     	~&0&1&0&1&1&1&0&0\cr
						 	                     	~&0&0&1&1&0&1&1&0\cr
					 	                     	    ~&1&1&1&0&0&0&0&1}$.   
                             \end{center}

Given a graph $H$, let $M(H)$ denote the circuit matroid of $H$.  A matroid $N$ is a \textit{single-element coextension} of a matroid $M$ if $N/e=M$ for some element $e$ of $N$.  
%  It follows that if $G$ is a graph and $T$ is a set of edges incident to a vertex  of $G$, then $M(G)_T'$ is graphic and $ M(G)_T'=M(G_T').$ 

Definition \ref{defel} is an extension of Definition \ref{nps} as $ M(G)_T'=M(G_T')$ for a set  $T$ of  edges incident to a vertex  of a graph $G.$ 
Note that if $M$ is a binary matroid, then the element splitting matroid $M_T'$ is  also binary and it is a coextension of $M$ by the element $a$ as $ M'_T /a = M .$    In  fact, we prove in Lemma \ref{elesingle} that every coextension of a binary matroid $M$ by a non-loop and non-coloop element is the element splitting matroid $M_T'$ for some $T \subset E(M).$

%the following result  in the second section. 
% that single-element coextension of a binary matroid $M$ by a non-loop and non-coloop element is nothing but the element splitting matroid $M_T'$ for some $T\subset E(M)$. 

 Dalvi et al. \cite{dalvi2009forbidden, dalvi2011forbidden}  characterized  the graphic (cographic) matroids $M$ whose single-element coextensions $M_T'$ are  again graphic (cographic).   Let $M$ be an $n$-connected binary matroid.  Borse and Mundhe \cite{bm}  obtained sufficient conditions for  the matroid $M_T'\backslash a$ to be $n$-connected.   In this paper, we obtain a necessary and sufficient condition for $M_T'$ to be $n$-connected.  The following is the main theorem of the paper.

  \begin{maintheorem}\label{maint}
  	Let $n\geq 2$ be an integer and $M$ be an $n$-connected binary matroid with $|E(M)| \geq 2n-2$. Suppose $T \subset E(M)$ with $|T|=n-1$.  Then  $M'_T$ is $n$-connected if and only if  $|Q|\geq 2|Q\cap T|$ for every cocircuit $Q$ of $M$ intersecting $T.$
  \end{maintheorem}

 We also prove that Theorem \ref{c5slater} follows from Main Theorem \ref{maint} under a mild restriction.  
 
 Azadi \cite{azadi2001generalized} obtained the following result for $M_T'$ to be $n$-connected, in terms of the circuits of $M$ containing an odd number of elements of $T$.
 
 \begin{theorem} [\cite{azadi2001generalized}] \label{c5azadith} 	Let $n\geq 2$ be an integer and $M$ be an $n$-connected binary matroid with $|E(M)| \geq 2n-2$. Suppose $T \subset E(M)$ with $|T|=n-1$. Then $M'_T$ is $n$-connected if and only if  for any set $A \subset E(M)$ with $|A|=n-2$, there exists a circuit $C$ of $M$ containing an odd number of elements of $T$ and is contained in $E(M)-A.$
 \end{theorem}
We  provide an alternate shorter proof of Theorem \ref{c5azadith} in the third section.

 In Section 2, we provide some properties of $M_T'.$   Main Theorem \ref{maint} is proved in  Section 3.  In the last section,  we discuss consequences of Main Theorem \ref{maint} to the graphs.

\section{Preliminaries} 

We prove below that the single-element coextension of a binary matroid $M$ by a non-loop and non-coloop element is nothing but an element splitting matroid $M_T'$ for some $T\subset E(M)$. 
\begin{lemma} \label{elesingle}
Let $M$ and $N$ be binary matroids. Then  $N$ is a  coextension of $M$ by a non-loop and non-coloop element if and only if $N=M_T'$  for some $T\subset E(M).$
\end{lemma}
\begin{proof}
Suppose $N=M_T'$ for some $T\subset E(M)$. Then the ground set of $N$ is $E(M) \cup \{a\}$  and $N/a=M$.  Hence $N$ is a coextension of $M$ by the element $a.$ Let $A$ be the standard matrix representation of $M$ over $GF(2)$. By Definition \ref{defel}, in the matrix $A_T'$ of $M_T',$  the column labeled by $a$ has 1 in the last row and 0 elsewhere, and the columns labeled by the elements of $T$ have 1 in the last row.  This shows that $a$ is neither a loop nor a coloop of $N.$

Conversely, suppose $N$ is a coextension of $M$ by a non-loop and non-coloop  element $a.$   Let $T_1$ be a cocircuit of $N$ containing $a$ and let $T=T_1-\{a\}$. Then $T$ is a non-empty subset of $E(M)$.  We can write the standard matrix representation  $B$ of $N$  such that the column of $B$ labeled by $a$ has entry 1 in the last row and  0 elsewhere.  Since $T_1$ is  a cocircuit of $N$,  the last row  of $B$ contains 1 in the columns corresponding to $T_1$ and 0 elsewhere.  Let $C$ be the matrix obtained from $B$ by deleting the last row and the column corresponding to $a$. Then $M[C]=N/a=M$.    Thus $B$ can be obtained from $C$ by adding one extra row which has entries 1 below the elements corresponding to $T$ and then adding a column labeled by $a$ which has entry 1 in the last row and 0 elsewhere. Therefore, by Definition \ref{defel},  $B=C_T'$. Hence $N=M[B]=M[C_T']=M_T'$.   
\end{proof}

 Henceforth, we use the notation $M_T'$ for a single-element coextension of a binary matroid $M.$

We need the following results.

%Azadi \cite{azadi2001generalized}  and Azanchiler \cite{azanchiler2005some} characterized the circuits and the rank function of matroid $M'_T$, respectively, as follows.
\begin{lemma} [\cite{azadi2001generalized}] \label{c5cesp}
	Let $M$ be a binary matroid and $T \subseteq E(M)$. If $\mathcal{C}$ is the collection of circuits of $M$, then every circuit of $M'_T$  belongs to one of the following type.
	\begin{enumerate} [(i).]
		
		\item  $\mathcal{C}_1=\{C\in \mathcal{C} \colon |C\cap T|~is~even \}$
		\item  $\mathcal{C}_2=\{C\cup \{a\} \colon C\in \mathcal{C}~and ~contains~an~odd ~number~of~elements~ of~T\}$
		\item  $\mathcal{C}_3=set~of~minimal~members~of~\{C_1\cup C_2 \colon C_1, C_2 \in \mathcal{C},~C_1\cap C_2= \emptyset~ ~and~C_1 ~and ~C_2~each~\\ contains~an~odd~number~of~ elements~of~T~such~ that~C_1\cup C_2~does ~not ~contain~any ~member~of~\mathcal{C}_1.\}$
	\end{enumerate}
	
\end{lemma}

\begin{lemma}[\cite{azanchiler2005some}] \label{c5res}
	Let $M$ be a binary matroid.  Suppose $r$ and $r'$ are the rank functions of $M$ and $M'_T$, respectively. If $A \subset E(M) \cup\{a\}$, then rank of $A$ is given by
	\begin{enumerate}[(i).]
		\item $r'(A)=r(A-\{a\})+1$ if $a \in A$.
		\item $r'(A)=r(A)+1$ if $a \notin A$ and $A$ contains a circuit $C$ of $M$ with $|C\cap T|$ odd.
		\item  $r'(A)=r(A)$ if $a\notin A$ and $A$ does not contain any circuit $C$ of $M$ with $|C\cap T|$  odd.
	\end{enumerate}
\end{lemma}

\begin{corollary}\label{c5rmt}
	Let $M$ be a binary matroid and $T \subseteq E(M)$. Then $r'(M_T')=r(M)+1$.
\end{corollary}

\begin{lemma} [\cite{azanchiler}] \label{c5cces} Let $M$ be a binary matroid and $\mathcal{C^*}$ be the collection of cocircuits of $M$. Suppose  $T\subseteq E(M)$ does not contain a cocircuit of $M$. Then every cocircuit of $M'_T$ belongs to one of the following type.
	\begin{enumerate} [(i).]
		\item	$\mathcal{Q}_1^*=\{ (C^*-T) \cup\{a\} \colon C^* \in \mathcal{C^*} ~and ~T~is~a~proper~subset~of~C^*\}$,
		\item	$\mathcal{Q}_2^*=\{C^* \colon C^* \in \mathcal{C^*}\}$,
		\item	$\mathcal{Q}_3^*=\{(C^* \Delta T)\cup\{a\} \colon C^* \in \mathcal{C^*}, 1\leq |C^* \cap T | < |T|~and~C^*~does~  not~contain~D^*-T~for ~any \\~D^* \in \mathcal{C^*}~and~T\subset D^* \},$
		\item	$\mathcal{Q}_4^*=\{((C^*_1 \cup C^*_2 \cup \dots \cup C^*_k)-T)\cup\{a\} \colon k\geq 2,~C^*_i \in \mathcal{C^*}, C^*_i\cap T \neq \emptyset, C^*_i ~are ~mutually~disjoint ~\\ and ~(C^*_1 \cup C^*_2 \cup \dots \cup C^*_k)-T~does~not~contain~D^*-T~for ~any~D^* \in \mathcal{C^*}~and~T\subset D^*\}$. 
		\item	$\mathcal{Q}_5^*=\{T\cup\{a\} \}$.
	\end{enumerate}
\end{lemma} 
%~and~C^*~does~not~contain~any ~member~of ~\mathcal{Q}_1^*
\section{Proofs} 
In this section, we prove Main Theorem \ref{maint}  and also provide an alternate shorter proof of Theorem \ref{c5azadith}.

We need the following result.
\begin{lemma} [\cite{oxley2006matroid}, pp 296]\label{11}
	If $n \geq 2$ and $M$ is an $n$-connected matroid with $|E(M)| \geq 2(n-1),$ then all circuits and all cocircuits of $M$ have at least $n$ elements. 
\end{lemma}

 Suppose $M$ is an $n$-connected binary matroid with $|E(M)| \geq 2(n-1)$ and $T \subset E(M).$  By Definition \ref{defel},  there is a cocircuit of $M_T'$ contained in $T \cup \{a\}.$  Therefore, if $|T|< n-1,$ then $M_T'$ contains a cocircuit of size less than $n$ by Lemma \ref{c5cces} and hence $M_T'$ is not $ n$-connected by Lemma \ref{11}. Hence we assume that $|T| \geq n-1.$

  We obtain below an obvious necessary condition for $M_T'$ to be $n$-connected. 
 \begin{lemma}\label{c7ifpartels}
	Let $n\geq 2$ be an integer and $M$ be an $n$-connected binary matroid with $|E(M)| \geq 2n-2$. Suppose $T \subset E(M)$ with $|T|=n-1$. If $M_T'$ is $n$-connected, then  $|Q| \geq 2|Q\cap T|$ for every cocircuit $Q$ of $M$ intersecting $T$.
	\end{lemma}
\begin{proof}
Suppose $M_T'$ is $n$-connected.  Assume that there is a cocircuit $Q$ of $M$ intersecting $T$ such that  $|Q| < 2|Q\cap T|.$  By Lemma \ref{c5cces} (iii),  $Q \Delta T\cup \{a\}$ contains a  cocircuit, say $X$, of $M'_T$.   Then
$|X| \leq |Q\Delta T\cup \{a\}|=|Q|+|T|-2|Q\cap T|+1 < |T|+1 = n,$
 a contradiction by Lemma \ref{11}.  \end{proof}

We now prove that the obvious necessary condition for $M_T'$ to be $n$-connected stated in the above lemma is sufficient also. 
\begin{proposition}  \label{c5pro}
			Let $n\geq 2$ be an integer and $M$ be an $n$-connected binary matroid with $|E(M)| \geq 2n-2$. Suppose $T \subset E(M)$ with $|T|=n-1$. If $|Q|\geq 2|Q\cap T|$ for every cocircuit $Q$ of $M$ intersecting $T$,   then $M'_T$ is $n$-connected.
\end{proposition} 
\begin{proof}
	Assume that $|Q|\geq 2|Q\cap T|$ for every cocircuit $Q$ of $M$ intersecting $T$. We proceed by contradiction. Suppose  $M'_T$ is not $n$-connected. Then there exists an $(n-1)$-separation $(A,B)$ of $M'_T.$ Therefore
	\begin{center} 	 min$\{|A|,|B|\} \geq n-1$ and 
		$r'(A)+r'(B)-r'(M'_T) \leq n-2.$ \ldots \ldots $(*)$
\end{center}

Suppose  $|A|\geq n$ and $|B| \geq n.$  Without loss of generality, we may assume that $a\in B$.  By Lemma \ref{c5res} and by $(*)$, 
$$r(A)+r(B-\{a\})-r(M) \leq r'(A)+r'(B)-1-(r'(M'_T)-1)\leq n-2.$$  	Therefore $(A, B-\{a\})$ forms an $(n-1)$-separation of $M$, a contradiction.

Therefore $|A|=n-1$ or $|B|=n-1.$  We may assume that $|A|=n-1.$   Then $A$ is independent in $M$ by Lemma \ref{11}.  Hence, by Lemma \ref{c5cesp},  $A$ is independent in $M'_T$ also.  

\vskip.2cm \noindent
\textit{Claim:  $A$ is a coindependent in $M'_T$.}
\vskip.2cm \noindent
Assume that $A$ is not coindependent in $M'_T$.  Then $A$ contains some cocircuit $Q$ of $M'_T$.  Therefore  $|Q| \leq |A|=n-1.$ By  Lemma \ref{11}, $Q$ is not a cocircuit of $M.$ Further, by  Lemma \ref{c5cces},  $Q$ does not belong to $\mathcal{Q}_2^*$. Hence $Q$    belongs to one of the four classes $\mathcal{Q}_1^*$, $\mathcal{Q}_3^*$, $\mathcal{Q}_4^*$ and $\mathcal{Q}_5^*.$

(1). Suppose $Q \in\mathcal{Q}_1^* .$  Then $Q = (C^* - T)\cup\{a\},$ where $C^*$ is a cocircuit of $M$ containing $T.$ Then,  by hypothesis, $|C^*| \geq 2|C^* \cap T|=2|T|=2n-2$. Therefore 
$$n- 1\geq |Q|=|C^*|-|T|+1\geq (2n-2)-(n-1)+1=n,$$ a contradiction.

(2). Suppose  $Q \in\mathcal{Q}_4^* .$  Then $Q =((C^*_1 \cup C^*_2 \cup \dots \cup C^*_k)-T)\cup\{a\},$  where $ k \geq 2$ and $C^*_i$ are mutually disjoint cocircuits of $M$ and each of them contains at least one element of  $T.$  Since $M$ is $n$-connected, $|C^*_i| \geq n$ for each $i$ by Lemma \ref{11}. Hence, we have 
$$|Q| \geq |(C^*_1 \cup C^*_2)- T | +1\geq |C^*_1| + |C^*_2| - |T| +1 \geq 2n -( n-1)+1 = n + 2 > n-1 \geq |Q|, $$
  again a contradiction.

(3). Suppose   $Q \in\mathcal{Q}_3^* .$  Then  $ Q = (C^* \Delta T)\cup\{a\},$ where $C^*$ is a cocircuit of $M$ intersecting $T.$ Hence 
$$ |Q| = |C^* \Delta T | +1= |C^*| + |T| - 2|C^*\cap T|+1 \geq |T|+1 = n > n-1 \geq |Q|,$$
 a contradiction.

(4). Suppose $ Q \in \mathcal{Q}_5^* $. So $Q=T\cup\{a\}.$ This gives $|Q|=n,$ a contradiction.

Thus in all the four cases, we get a contradiction.  This proves the claim.

Therefore  $A$ is independent and coindependent in the matroid $M'_T.$  Hence $r'(A)=|A|$ and $r'(B) = r'(M'_T).$  This gives $n-1=|A|=r'(A)=r'(A)+r'(B)-r'(M'_T) \leq n-2, $	
a contradiction.  Thus we get a contradiction in each case.  Therefore $M'_T$ is  $n$-connected. \end{proof}
%The above Proposition is extension of the corresponding graph result stated in Theorem \ref{c5slater}.

Main Theorem \ref{maint} follows obviously from  Lemma \ref{c7ifpartels} and Proposition \ref{c5pro}.

For $2\leq n \leq 4$, we get the following weaker sufficient conditions for $M_T'$ to be $n$-connected. 
\begin{corollary}\label{c7cornes}
Let $n\in \{2,3,4\}$ and let $M$ be $n$-connected binary matroid.  Suppose $T\subset E(M)$ with $|T|=n-1$. If $|Q| \geq 2n-2$ for every  cocircuit $Q$ containing $T$,  then $M'_T$ is $n$-connected. 
\end{corollary}
\begin{proof}
Let $Q$ be a cocircuit of $M$ intersecting $T$. By Proposition \ref{c5pro}, it is sufficient to prove that $|Q|\geq 2|Q\cap T|$. If $T \subseteq Q$, then $|Q|\geq 2n-2=2|T|=2|Q\cap T|$. Suppose $T \nsubseteq Q$. Then $|Q\cap T|<|T|=n-1$ and hence $|Q\cap T|\leq n-2$.  Since $ 2\leq n \leq 4$, we have $2|Q\cap T| \leq 2(n-2)=2n-4 \leq n.$ By Lemma \ref{11}, $|Q| \geq n$ and so $|Q| \geq 2|Q\cap T|$. 
\end{proof}

We combine Main Theorem \ref{maint}  and Theorem \ref{c5azadith} and provide a shorter proof of Theorem \ref{c5azadith}. 

 \begin{theorem}
 		Let $n\geq 2$ be an integer and $M$ be an $n$-connected binary matroid with $|E(M)| \geq 2n-2$. Suppose $T \subset E(M)$ with $|T|=n-1$. Then the following statements are equivalent. 
 
 	\begin{enumerate}[(i).]
 		\item $M'_T$ is $n$-connected.
 		\item $|Q|\geq 2|Q\cap T|$ for every cocircuit $Q$ of $M$ intersecting $T.$  
 		\item For any subset $A \subset E(M)$ with $|A|=n-2$, there exists a circuit $C$ of $M$ containing an odd number of elements of $T$ and is contained in $E(M)-A.$
 	\end{enumerate}
 \end{theorem}
\begin{proof} (i) $\implies$ (ii) follows from Lemma \ref{c7ifpartels} and (ii) $\implies$ (i) follows from Proposition \ref{c5pro}.

(i) $\implies$ (iii).  Suppose (i) holds  but (iii) does not hold.    Then there is a subset $A$ of $E(M)$ with $|A|=n-2$ such that  no circuit of $M$ containing an odd number of elements of $T$ is contained in $E(M)-A.$  	Let $A'=A \cup \{a\}$ and $B=E(M)-A.$  Then $|A'|= n-1$ and $|B|\geq n-1.$ Let $r$ and $r'$ be the rank function of $M$ and  $M_T'$, respectively. By Lemma \ref{11},  $A$ contains neither a cocircuit of $M$ nor a cocircuit of $M'_T.$ Hence $r(B)=r(M)$ and $r'(B)=r'(M_T')$.   Also, by Lemma \ref{c5res}(iii), $r(B)=r'(B).$ This gives $r(M)=r'(M_T')$, a contradiction by Corollary \ref{c5rmt}. Hence (i) implies (iii).

	(iii) $\implies$ (i). 	Suppose (iii) holds  but (i) does not hold. Then $M_T'$ has an $(n-1)$-separation $(A,B).$    Therefore
	\begin{center} 
	 min$\{|A|,|B|\} \geq n-1$ and	  $r'(A)-r'(B)-r'(M'_T)\leq n-2$. \ldots \ldots $(*)$
	  \end{center} 
  
	Without loss of generality, assume that $a\in A.$  By Lemma \ref{c5res}(i), $r'(A)=r(A-\{a\})+1$.  If $|A|\geq n$,  then, by $(*)$, 
	$$r(A-\{a\})+r(B)-r(M) \leq r'(A)-1+r'(B)-(r'(M'_T)-1)\leq n-2.$$
	  Therefore $(A-\{a\}, B)$ is an $(n-1)$-separation of $M$, a contradiction.
	Hence  $|A|=n-1$. Then $|A-\{a\}|=n-2$.  By (iii) and Lemma \ref{c5res}(ii),  $r'(B)=r(B)+1$. 	Therefore 	
	$$r(A-\{a\})+r(B)-r(M)\leq r'(A)-1+r'(B)-1-(r'(M'_T)-1)=n-3.$$
	  This shows that  $(A,B)$ is an $(n-2)$-separation of $M,$ a contradiction. Thus (iii) implies (i). \end{proof}

\section{Consequences to Graphs}
In this section, we prove that Proposition \ref{c5pro} is a matroid extension of  Theorem \ref{c5slater}. 
 
We need the following result. 
\begin{theorem} [\cite{oxley2006matroid}, pp. 328] \label{c1gnmn}
	For $n \geq 2$, let $G$ be a graph without isolated vertices and with at least $n+1$ vertices.  Then the circuit matroid $M(G)$ is $n$-connected if and only if $G$ is $n$-connected and has no cycle with fewer than $n$ edges. 
\end{theorem} 

 By Theorem \ref{c1gnmn},   the circuit matroid $M(G)$  of an $n$-connected  graph $G$ is not $n$-connected if  $G$ contain a cycle of length  less than $n$.  Therefore we derive Theorem \ref{c5slater} from Proposition \ref{c5pro} by assuming that $G$ has girth at least $n$.  

\begin{theorem}\label{c7gtth}
 Suppose $G$ is an 	$n$-connected graph of girth at least $n$, where $n\geq 2$. Let $T$ be a set of $n-1$ edges incident to a vertex of degree at least $2n-2$ in $G$. Then the $n$-point splitting graph $G_T'$ is $n$-connected.  
\end{theorem}
\begin{proof} 	
Let $M=M(G)$. Then $M_T'=M(G_T')$. We prove that $M_T'$ is $n$-connected.  By Theorem \ref{c1gnmn}, $M$ is $n$-connected.  Let $Q$ be a cocircuit of $M$ intersecting $T$. By Proposition  \ref{c5pro}, it is sufficient to prove that $|Q| \geq 2|Q\cap T|$.  On the contrary, assume that  $|Q| < 2|Q\cap T|$.  	As $Q=(Q-T) \cup (Q \cap T)$, 	$|Q|=  \cfrac{|Q|}{2}  +\cfrac{|Q|}{2}=|Q-T|+|Q\cap T|$ and hence	$|Q-T| < \cfrac{|Q|}{2} <|Q\cap T|.$  Let $u$ be the vertex of $G$ of degree at least $2n-2$ such that the edges of $G$ belonging to $T$ are incident to $u$. 	Since $Q$ is a cocircuit of $M(G)$, the graph $G-Q$ is disconnected and it has two components, say $C_1$ and $C_2$.  We may assume that $C_2$ contains the vertex $u$.  Let $Q\cap T =\{uu_1,uu_2,\dots, uu_k\}.$ Then $u_1, u_2, \dots, u_k$ are vertices of $C_1$. Let $v_1,v_2, \dots, v_r$ be the end vertices of the edges belonging to $Q-T$ in $C_1$. Then $r \leq |Q-T|<|Q\cap T|.$  Since $|Q| < 2|Q\cap T| \leq 2|T| =2n-2$ and degree of $u$ is at least $2n-2$,  there is at least one edge $uw$ incident to $u$ in $G-Q$.  Then the edge $uw$ is in $C_2$.  Let $A=\{v_1,v_2, \dots, v_r, u\}$. Then $G-A$ is disconnected, leaving $u_i$ for some $i\in \{1,2,\dots, k\}$ in one component and the vertex $w$ is in an another component.   However,  $|A|= r+1\leq |Q \cap T| \leq |T|=n-1,$   a contradiction to the fact that $G$ is $n$-connected.  Thus  $M_T'$ is $n$-connected. By Theorem \ref{c1gnmn}, $G_T'$ is $n$-connected.  \end{proof}

%If $G$ is a simple $3$-connected graph, then it does not contain a cycle of length less than three and vertex of degree less than $3$.  Therefore we have the following result.

\begin{corollary} \label{c5cors}\label{c5corr}
Let $G$ be a $3$-connected  simple graph and $T$ be a set of two edges incident to a vertex  of $G$ of degree at least four. Then the graph $G'_T$ is $3$-connected. 
\end{corollary}

We now prove that one can obtain a 3-regular, 3-connected graph from the given 3-connected simple graph by repeated applications of 3-point splitting operation.
\begin{corollary} A 3-regular, $3$-connected simple graph can be obtained from the given $3$-connected simple graph by a finite sequence of the 3-point splitting operation.
\end{corollary}

\begin{proof} Let $G$ be a $3$-connected simple graph. Then degree of every vertex of $G$ is at least three. Suppose  $G$ contains a vertex $v$ of degree $ k > 3.$ Let $T= \{x, y\}$ be a set of two edges incident at $v$.  By Corollary \ref{c5corr}, $G'_T$ is $3$-connected. 	The vertex of $v$ of $G$ is replaced by two vertices $v'$  and $v''$ with degrees $3$ and $k-1,$ respectively in $G_T'$.  Thus one application of  3-point splitting on a vertex  of degree $t>3$ results into a 3-connected graph with one additional vertex of degree less than $t$. By a finite sequence of  3-point splitting operation we can get a 3-connected graph with no vertex of degree greater than three.  Clearly, this graph will be 3-regular. \end{proof}

\end{document}